\documentclass[sigconf,nonacm]{acmart}

\AtBeginDocument{%
  }

\usepackage{algorithm}
\usepackage{algpseudocode}
\usepackage{amsmath}
\usepackage{amsthm}
\usepackage{bm}
\usepackage[subrefformat=parens]{subcaption}
\usepackage{mathrsfs}
\usepackage{mathtools}
\usepackage{float}
\usepackage{comment}
\theoremstyle{plain}
\newtheorem{proposition}{Proposition}
\usepackage{booktabs}
\usepackage{siunitx}
\usepackage{url}
\usepackage{multirow}

\setcopyright{none}
\renewcommand\footnotetextcopyrightpermission[1]{}

\begin{document}

\title{Safe screening rules for portfolio optimization \\with linear and cardinality constraints}

\author{Nanari Wada}
\affiliation{%
  \institution{University of Tsukuba}
  \city{Tsukuba}
  \state{Ibaraki}
  \country{Japan}}
\email{s2520497@u.tsukuba.ac.jp}

\author{Shunnosuke Ikeda}
\affiliation{%
  \institution{University of Tsukuba}
  \city{Tsukuba}
  \state{Ibaraki}
  \country{Japan}
}
\email{ikeda@cs.tsukuba.ac.jp}

\author{Yuichi Takano}
\affiliation{%
 \institution{University of Tsukuba}
 \city{Tsukuba}
 \state{Ibaraki}
 \country{Japan}
}
\email{ytakano@sk.tsukuba.ac.jp}

\author{Jun-ya Gotoh}
\affiliation{%
 \institution{Chuo University}
 \city{Bunkyo-ku}
 \state{Tokyo}
 \country{Japan}
}
\email{jgoto@kc.chuo-u.ac.jp}

\renewcommand{\shortauthors}{Wada et al.}

\begin{abstract}
In portfolio optimization, a cardinality constraint, which limits the number of assets held, plays a key role in cutting down monitoring and transaction costs.
However, the resulting problem is NP-hard and becomes computationally difficult to solve globally as the number of candidate assets grows. 
Safe screening addresses this difficulty by fixing decision variables before optimization without excluding any globally optimal solution, thereby reducing the problem size while preserving optimality guarantees. 
We propose safe screening rules for cardinality-constrained portfolio optimization with a convex quadratic objective function and linear constraints. 
Using a perspective relaxation of the $\ell_2$-regularization term and Fenchel duality, we derive asset-specific scores that incorporate the Lagrange multipliers of the linear constraints. 
Combined with a relaxation-based lower bound and a feasible-solution upper bound, these scores safely fix binary asset-selection variables to zero or one. 
Experiments on S\&P 500 and Russell 2000 datasets show substantial computational improvements on challenging cases, particularly under moderate or strong regularization and less stringent return requirements. 
These results demonstrate the effectiveness of safe screening as an optimality-preserving preprocessing technique that greatly boosts computational efficiency in large-scale cardinality-constrained portfolio optimization.
\end{abstract}


\maketitle
\section{Introduction}
\subsection{Background}
Portfolio optimization is the problem of determining how to allocate investments across multiple financial assets while accounting for factors such as risk and expected return.
Since the seminal work of Markowitz~\cite{markowitz1952portfolio},
portfolio optimization has developed into a broad research field encompassing various risk measures, estimation uncertainty, transaction costs, practical investment constraints, and computational solution methods~\cite{gunjan2023brief,salo2024fifty,Lee2024}.
In particular, large-scale portfolio optimization has been recognized as an important computational problem in practical investment management~\cite{perold1984large}.

As the number of assets included in a portfolio increases, the associated monitoring and transaction costs also tend to increase~\cite{kolm2014sixty}.
Therefore, in practical applications, it is often necessary to control the number of assets selected for investment~\cite{bertsimas1999portfolio}.
This requirement gives rise to cardinality-constrained portfolio optimization models, which determine an optimal asset allocation subject to an upper bound on the number of nonzero elements in the portfolio weight vector \cite{bienstock1996computational,gao2013optimal}.

However, cardinality-constrained portfolio optimization is known to be NP-hard~\cite{gao2013optimal}, making it computationally difficult to obtain an optimal solution when the number of candidate assets is large~\cite{bienstock1996computational}.
Although recent advances in mixed-integer optimization solvers have facilitated efforts to solve such problems to proven optimality \cite{tillmann2024cardinality}, the computational burden remains substantial for large-scale instances.
Therefore, there is a strong demand for efficient preprocessing techniques to accelerate the solution of cardinality-constrained portfolio optimization problems.

\subsection{Related Work}
Cardinality-constrained portfolio optimization has been studied using a variety of solution approaches.
Chang et al.~\cite{chang2000heuristics} developed heuristic methods for constructing the cardinality-constrained efficient frontier, whereas Bienstock~\cite{bienstock1996computational} investigated mixed-integer quadratic optimization formulations.
Zheng et al.~\cite{zheng2014improving} used semidefinite optimization to compute a diagonal decomposition that tightens the continuous relaxation of the perspective reformulation.
Bertsimas and Cory-Wright~\cite{bertsimas2022scalable} proposed a
scalable cutting-plane algorithm for sparse portfolio selection and summarized the scalability limitations of existing approaches.
Kobayashi et al.~\cite{kobayashi2021bilevel} developed a bilevel cutting-plane algorithm for exactly solving
cardinality-constrained mean-CVaR portfolio optimization problems, and Kobayashi et al.~\cite{kobayashi2023cardinality} proposed a
cutting-plane algorithm combined with positive semidefinite matrix completion for cardinality-constrained distributionally robust portfolio optimization.
Although these studies have significantly advanced the solution of cardinality-constrained portfolio optimization problems, exact optimization remains computationally demanding for large-scale instances.

Safe screening is a preprocessing technique that fixes decision variables without excluding any optimal solution.
It has been extensively studied in the field of machine learning with sparsity.
El Ghaoui et al.~\cite{elghaoui2012safe} proposed safe feature elimination rules for the Lasso and related sparse learning problems.
Fercoq et al.~\cite{fercoq2015mind} subsequently developed duality-gap-based safe screening rules that become increasingly effective as the optimization algorithm converges.
This paradigm has recently been extended to cardinality-constrained sparse learning.
Atamt{\"u}rk and G{\'o}mez~\cite{atamturk2020safe} proposed safe screening rules for cardinality-constrained linear regression based on perspective relaxation and Fenchel duality.
This framework for safe screening has been applied to logistic regression by Deza and Atamt{\"u}rk~\cite{deza2022safe} and to Poisson regression by Kurihara and Izunaga~\cite{kurihara2025scalable}.
However, these studies primarily address sparse learning models without additional linear equality or inequality constraints.
They therefore do not cover portfolio optimization models that incorporate practical linear constraints, such as budget and nonnegativity constraints.

A related line of research has considered safe screening in regularized portfolio optimization.
Du et al.~\cite{du2023high} proposed a method combining
$\ell_1$ regularization with screening for high-dimensional portfolio optimization.
Their study demonstrates the potential effectiveness of preprocessing for reducing the number of candidate assets in portfolio selection.
However, it does not address safe screening for the
cardinality-constrained portfolio optimization considered in this paper.

\subsection{Contribution}
Building on Atamt{\"u}rk and G{\'o}mez~\cite{atamturk2020safe}, we extend safe screening rules based on perspective relaxation for
cardinality-constrained sparse regression to cardinality-constrained portfolio optimization with linear constraints.
This extension is nontrivial because it deviates from existing rules due to the objective functions and constraints involved in realistic portfolio optimization.
Our main contribution is the derivation of asset-specific screening scores and safe variable-fixing rules that explicitly incorporate additional Lagrange multipliers in the Fenchel dual optimal solution.
The resulting framework applies to cardinality-constrained portfolio optimization models with convex quadratic objectives and practical linear constraints.

We evaluate our screening rules on problem instances constructed from S\&P 500 and Russell 2000 datasets.
The computational results show substantial improvements on challenging instances, particularly under moderate or strong regularization and less stringent return requirements.
For the S\&P 500 dataset, our method frequently achieved speedups of more than tenfold for problem instances that require relatively long computation times, while for the Russell 2000 dataset, the reduced problems were solved to global optimality in several cases where direct optimization reached the time limit.
These results support safe screening as an effective optimality-preserving preprocessing technique for large-scale cardinality-constrained portfolio optimization.

\section{Problem Formulation}
In this section, we formulate the cardinality-constrained portfolio optimization problem.
We then present its mixed-binary optimization reformulation.
Throughout this paper, we denote the set of consecutive positive integers as $[n] := \{1,2,\dots,n\}$.
\subsection{Cardinality-Constrained Portfolio Optimization Problem}

We consider the following class of cardinality-constrained portfolio
optimization problems:
\begin{subequations}\label{original}
\begin{align}
\min_{\bm{x}\in\mathbb{R}^{n}}
\quad &
\frac{1}{2}\bm{x}^{\top}\bm{Q}\bm{x}
+\bm{c}^{\top}\bm{x}
+\tau\|\bm{x}\|_{2}^{2}
\label{obj:original}
\\
\mathrm{s.~t.}\quad
&\|\bm{x}\|_{0}\le\kappa,
\label{cons:original-card}\\
&\bm{A}\bm{x}=\bm{b},
\label{cons:original-equality}\\
&\bm{W}\bm{x}\ge\bm{h}.
\label{cons:original-neq}
\end{align}
\end{subequations}
Here, $\bm{x}\in\mathbb{R}^{n}$ is a vector of decision variables representing the portfolio weights of
$n$ candidate assets, and $\bm{Q}\in\mathbb{S}_{+}^{n}$ and $\bm{c}\in\mathbb{R}^{n}$
define the convex quadratic objective.
The matrices $\bm{A}\in\mathbb{R}^{m_1\times n}$ and $\bm{W}\in\mathbb{R}^{m_2\times n}$, together with the vectors $\bm{b}\in\mathbb{R}^{m_1}$ and $\bm{h}\in\mathbb{R}^{m_2}$, define the linear equality and inequality constraints, respectively.
The parameter $\kappa\in[n]$ is the upper bound on the number of assets included in the portfolio, and
$\tau>0$ is the regularization parameter.
The objective function is the sum of a convex quadratic function and an $\ell_2$-regularization term, and encompasses, among others, mean--variance~\cite{gao2013optimal} and index-tracking models~\cite{takeda2013simultaneous}.
From a practical perspective, the $\ell_2$-regularization term reduces the sensitivity of portfolio weights to estimation error and can thereby improve out-of-sample performance~\cite{demiguel2009generalized,bertsimas2022scalable,gotoh2013robust,gotoh2011role}.
Eq.~\eqref{cons:original-card} is a cardinality constraint that limits the number of assets included in the portfolio.
Eq.~\eqref{cons:original-equality} represents linear equality constraints, including the budget constraint requiring the portfolio weights to sum to one. Eq.~\eqref{cons:original-neq} represents linear inequality constraints, including constraints on expected returns and portfolio weights.

\subsection{Mixed-binary Optimization Reformulation}

We introduce a binary decision variable $\bm{z}:= (z_i)_{i\in [n]} \in \{0,1\}^n$ indicating the selection of investment assets.
Here, $z_i=1$ indicates that asset $i$ is included in the investment universe, whereas $z_i=0$ indicates that asset $i$ is excluded from it.
The relationship between the portfolio weight $x_i$ and the binary indicator $z_i$ is represented by the following logical implication constraint:
\begin{equation}
z_i = 0 \Rightarrow x_i = 0,
\quad i\in[n].
\label{logical_imp}
\end{equation}
Eq.~\eqref{logical_imp} ensures that $x_i=0$ whenever $z_i=0$ and, equivalently, that $z_i=1$ whenever $x_i\neq 0$.

Combining the logical implication (Eq.~\eqref{logical_imp}) with the cardinality constraint $\sum_{i=1}^{n}z_i\leq\kappa$, the original problem~\eqref{original} can be
equivalently reformulated as the following mixed-binary optimization problem:
\begin{subequations} \label{binary-reformulation}
\begin{align}
    \min_{\bm{x},\bm{z}}
    \quad &
    \frac{1}{2}\bm{x}^{\top}\bm{Q}\bm{x}
    +\bm{c}^{\top}\bm{x}
    +\tau\lVert\bm{x}\rVert_2^2
    \label{prob:binary-reformulation}
    \\
    \text{s.~t.}\quad
    &
    \bm{A}\bm{x}=\bm{b},
    \quad
    \bm{W}\bm{x} \ge \bm{h},\label{cons:reform-linear}
    \\
    &
    \sum_{i=1}^{n}z_i\leq\kappa,
    \label{cons:reform-cardinality}
    \\
    &
    z_i = 0 \Rightarrow x_i = 0
    \quad (i \in [n]),
    \label{cons:reform-link}
    \\
    &
    \bm{z}\in\{0,1\}^{n}
    \label{cons:reform-binary}.
\end{align}
\end{subequations}

\section{Safe Screening Rules}
In this section, we develop safe screening rules for the
cardinality-constrained portfolio optimization problem~\eqref{original}.
We first explain the basic principle of safe screening and then
derive asset-specific screening scores using a perspective relaxation and Fenchel duality.
Based on these scores, we establish safe exclusion and inclusion
rules.
Finally, we present the screening procedure in an algorithm.
\subsection{Basic Principle of Safe Screening}
Safe screening fixes binary variables while preserving optimality by comparing a valid upper bound with lower bounds derived from subproblems in which individual binary variables are fixed to either zero or one.
Let $f^\star$ denote the optimal objective value of the original problem~\eqref{binary-reformulation}, let $f_{\mathrm{LB}}$ be a lower bound obtained from a relaxation, and let $f_{\mathrm{UB}}$ be the objective value of a feasible solution.
It then follows that
\begin{equation}
f_{\mathrm{LB}}
\leq
f^\star
\leq
f_{\mathrm{UB}}.
\label{eq:basic-bounds}
\end{equation}

For each $i\in[n]$, if a lower bound obtained after imposing $z_i=1$
exceeds $f_{\mathrm{UB}}$, then no optimal solution can satisfy $z_i=1$,
and hence $z_i$ can be fixed to zero.
Similarly, if a lower bound obtained after imposing $z_i=0$
exceeds $f_{\mathrm{UB}}$, then $z_i$ can be fixed to one.
Therefore,
\begin{align}
&
f_{\text{LB}}\text{ with } z_i=1
>
f_{\mathrm{UB}}
\quad\Longrightarrow\quad
z_i=0,
\label{eq:basic-screening-zero}
\\
&
f_{\text{LB}}\text{ with } z_i=0
>
f_{\mathrm{UB}}
\quad\Longrightarrow\quad
z_i=1.
\label{eq:basic-screening-one}
\end{align}

\begin{figure}[t]
    \centering
    \captionsetup{skip=3pt}
    \includegraphics[width=\columnwidth]{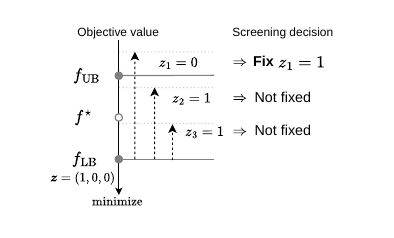}
    \caption{Basic principle of safe screening.}
    \Description{
        Illustration of safe screening based on the comparison between
        a global upper bound and lower bounds obtained by fixing binary
        variables.
    }
    \label{fig:safe-screening-concept}
     \vspace{-10pt}
\end{figure}
Figure~\ref{fig:safe-screening-concept} illustrates the concept of safe screening.
The vector $\bm z=(1,0,0)$ represents the asset-selection pattern that attains the lower-bound value $f_{\mathrm{LB}}$ in the relaxation problem.
Adding a fixing constraint restricts the feasible region of the relaxation problem, and hence the optimal objective value of the resulting restricted relaxation will be larger than $f_{\mathrm{LB}}$.
In the example shown in the figure, the lower bound obtained by imposing $z_1=0$ exceeds $f_{\mathrm{UB}}$.
Since $f_{\mathrm{UB}}$ is a valid upper bound on the optimal objective value of the original problem~\eqref{binary-reformulation}, no solution satisfying $z_1=0$ can be optimal.
Therefore, $z_1$ can be safely fixed to one.
In contrast, the lower bounds obtained by imposing $z_2=1$ and $z_3=1$ do not exceed $f_{\mathrm{UB}}$.
Hence, these particular tests are inconclusive and do not allow $z_2$ or $z_3$ to be fixed.

\subsection{Convex Relaxation Using Perspective Functions}

A standard approach to formulating the logical implication (Eq.~\eqref{logical_imp}) between asset selection and portfolio weights is to use the Big-$M$ constraints $-M z_i \le x_i \le M z_i$, where $M$ is a sufficiently large constant~\cite{bertsimas2022scalable}.
However, Big-$M$ formulations may yield weak continuous relaxations, and excessively large values of $M$ may also cause numerical difficulties~\cite{tillmann2024cardinality,belotti2016handling}.

To obtain a stronger convex relaxation, we apply a perspective
reformulation~\cite{gunluk2010perspective} to the $\ell_2$-regularization term.
For each $i\in[n]$, we define
\begin{equation}
    \phi(x_i,z_i):=
    \begin{cases}
        x_i^2/z_i, & z_i>0,\\
        0, & x_i=0,\ z_i=0,\\
        +\infty, & x_i \neq 0,\ z_i = 0.
    \end{cases}
\end{equation}
The function $\phi$ is the perspective of the convex function and is therefore convex~\cite{HiriartUrrutyLemarechal1993}.
Then, problem~\eqref{binary-reformulation} is equivalently reformulated as
\begin{subequations} \label{perspective}
\begin{align}
    \min_{\bm{x},\bm{z}}
    \quad&
    \frac{1}{2}\bm{x}^{\top}\bm{Q}\bm{x}
    +\bm{c}^{\top}\bm{x}
    +\tau\sum_{i=1}^{n}\phi(x_i,z_i)
    \label{prob:perspective-integer}
    \\
    \text{s.~t.}\quad&
    \bm{A}\bm{x}=\bm{b},\quad
    \bm{W}\bm{x}\geq\bm{h},\label{cons:perspective-linear}
    \\
    &
    \sum_{i=1}^{n}z_i\leq\kappa,\label{cons:perspective-card}
    \\
    &
    \bm{z}\in\{0,1\}^{n}.
    \label{cons:perspective-integer}
\end{align}
\end{subequations}
Indeed, when $z_i=1$, we have
$\phi(x_i,1)=x_i^2$, which coincides with the \(i\)th quadratic
component of the original $\ell_2$-regularization term.
When $z_i=0$, the objective value is finite only if $x_i=0$; therefore, the logical implication (Eq.~\eqref{cons:reform-link}) holds.
Therefore, problem~\eqref{perspective} is equivalent to the mixed-binary optimization problem~\eqref{binary-reformulation}.

Relaxing \(\bm{z} \in\{0, 1\}^n\) to \(\bm{ z} \in[0, 1]^n\) 
yields the following convex relaxation:
\begin{subequations} \label{relaxation}
\begin{align}
    \min_{\bm{x},\bm{z}}
    \quad&
    \frac{1}{2}\bm{x}^{\top}\bm{Q}\bm{x}
    +\bm{c}^{\top}\bm{x}
    +\tau\sum_{i=1}^{n}\phi(x_i,z_i)
    \label{prob:relaxation}
    \\
    \text{s.~t.}\quad&
    \text{Eqs.}~\eqref{cons:perspective-linear}\text{--}\eqref{cons:perspective-card}
    \\
    &
    \bm{z}\in[0,1]^n.
\end{align}
\end{subequations}
This problem is convex and provides a valid lower bound on the
original problem~\eqref{binary-reformulation}.
Moreover, if it is feasible, then it admits an optimal solution.
For later use, we define the feasible sets for $\bm{x}$ and $\bm{z}$,
respectively, as
\begin{align}
    \mathcal{X}
    &:=
    \left\{
        \bm{x}\in\mathbb{R}^n
        \;\middle|\;
        \bm{A}\bm{x}=\bm{b},\
        \bm{W}\bm{x}\geq\bm{h}
    \right\},
    \\
    \mathcal{Z}
    &:=
    \left\{
        \bm{z}\in[0,1]^n
        \;\middle|\;
        \bm{1}^{\top}\bm{z}\leq\kappa
    \right\}.
\end{align}
Accordingly, the feasible set of the convex relaxation
problem~\eqref{relaxation} is
\begin{equation}
    \mathcal{F}
    :=
    \mathcal{X}\times\mathcal{Z}.
\end{equation}

We also give a conic representation of the perspective formulation~\cite{frangioni2006perspective,gunluk2010perspective,atamturk2020safe}.
We introduce auxiliary variable $\bm{t}:=(t_i)_{i\in [n]}\in \mathbb{R}_{+}^n$ for representing perspective functions.
The resulting conic representation of problem~\eqref{relaxation} is given by
\begin{subequations}
\label{relaxation-soc}
\begin{align}
    \min_{\bm t,\bm x,\bm z}\quad
    & \frac{1}{2}\bm x^\top\bm Q\bm x
      + \bm c^\top\bm x
      + \tau\sum_{i=1}^{n} t_i \\
    \mathrm{s.~t.}\quad
    &
    \bm{A}\bm{x}=\bm{b},\quad
    \bm{W}\bm{x}\geq\bm{h},\label{cons:soc-linear}
    \\
    & x_i^2\leq t_i z_i\quad (i \in [n]),
      \label{cons:perspective-relaxation-soc} \\
    & \bm{t}\geq\bm 0,
    \\&
    \sum_{i=1}^{n}z_i\leq\kappa,\label{cons:soc-card}
    \\
    &
    \bm{z}\in[0,1]^n.\label{cons:perspective-relaxation-domain}
\end{align}
\end{subequations}
The inequalities in Eq.~\eqref{cons:perspective-relaxation-soc} are representable by rotated second-order cones~\cite{alizadeh2003second}.
The corresponding mixed-binary optimization formulation is obtained by replacing $\bm{z}\in[0,1]^n$ with $\bm z\in\{0,1\}^n$ as
\begin{subequations}
\label{binary-soc}
\begin{align}
    \min_{\bm t,\bm x,\bm z}\quad
    & \frac{1}{2}\bm x^\top\bm Q\bm x
      + \bm c^\top\bm x
      + \tau\sum_{i=1}^{n} t_i \\
    \mathrm{s.t.}\quad
    &
    \text{Eqs.}~\eqref{cons:soc-linear}\text{--}\eqref{cons:soc-card}
    \\
    &
    \bm{z}\in\{0,1\}^n.
\end{align}
\end{subequations}

\subsection{Fenchel Dual Problem}

Using the conjugate representation of the perspective
function~\cite{atamturk2020safe}, we have
\begin{equation}
\phi(x_i,z_i)
=
\max_{p_i\in\mathbb{R}} \left\{p_i x_i-\frac{p_i^2}{4}z_i\right\}.
\label{eq:perspective-fenchel-representation}
\end{equation}
Accordingly, we define
\begin{equation}
    \mathcal{L}_{F}
    (\bm{x},\bm{z},\bm{p})
    :=
    \frac{1}{2}\bm{x}^{\top}\bm{Q}\bm{x}
    +\bm{c}^{\top}\bm{x}
    +
    \tau\sum_{i=1}^{n}
    \left(
        p_i x_i-\frac{p_i^2}{4}z_i
    \right).\label{obj:fenchel}
\end{equation}
By Eq.~\eqref{eq:perspective-fenchel-representation}, the optimal objective value of the convex relaxation problem~\eqref{relaxation} can be written as
\begin{equation}
    f_{\mathrm{LB}}
    =
    \min_{(\bm{x},\bm{z})\in\mathcal{F}}
    \max_{\bm{p}\in\mathbb{R}^n}
    \mathcal{L}_{F}
    (\bm{x},\bm{z},\bm{p}).
\end{equation}

Since the convex relaxation is feasible and $1\leq\kappa<n$, the relative-interior condition for strong duality is satisfied.
Hence, the strong duality theorem for convex optimization~\cite{rockafellar1970convex} guarantees attainment of a dual optimal solution and allows the order of optimization to be interchanged.
The resulting Fenchel dual problem associated with the convex relaxation problem~\eqref{relaxation} is
\begin{equation}
f_{\mathrm{LB}}
=
\max_{\bm{p}\in\mathbb{R}^n}
\min_{(\bm{x},\bm{z})\in\mathcal{F}}
\mathcal{L}_{F}(\bm{x},\bm{z},\bm{p}).
\label{prob:fenchel-dual}
\end{equation}

\subsection{Screening Scores}

Let $\hat{\bm{p}}$ be an optimal solution to the Fenchel dual problem~\eqref{prob:fenchel-dual}, and let $(\hat{\bm{x}},\hat{\bm{z}})$ be an optimal solution to the convex relaxation problem~\eqref{relaxation}.
By strong duality and attainment of both optimal solutions, $(\hat{\bm x},\hat{\bm z})$ is also an optimal solution to problem~\eqref{prob:fenchel-dual}.
Since both the objective function and the feasible set separate with respect to $\bm{x}$ and $\bm{z}$, the Fenchel dual problem~\eqref{prob:fenchel-dual} can be decomposed into independent subproblems in $\bm{x}$ and $\bm{z}$.

We first consider the subproblem with respect to $\bm{x} \in \mathbb{R}^n$.
Let $\bm{\lambda}\in\mathbb{R}^{m_1}$ and
$\bm{\nu}\in\mathbb{R}_{+}^{m_2}$ denote the Lagrange multipliers associated with
$\bm{A}\bm{x}=\bm{b}$ and $\bm{W}\bm{x}\geq\bm{h}$, respectively.
Its Lagrangian is
\begin{align}
    \mathcal{L}_{\bm{x}}
    \left(
        \bm{x},\bm{\lambda},\bm{\nu};\bm{p}
    \right)
    :={}&
    \frac{1}{2}\bm{x}^{\top}\bm{Q}\bm{x}
    +\bm{c}^{\top}\bm{x}
    +\tau\bm{p}^{\top}\bm{x}
    \notag\\
    &+
    \bm{\lambda}^{\top}(\bm{A}\bm{x}-\bm{b})
    -
    \bm{\nu}^{\top}(\bm{W}\bm{x}-\bm{h}).
\end{align}
Since the feasible region of this subproblem is polyhedral, there exist Lagrange multipliers
$\hat{\bm{\lambda}}$ and $\hat{\bm{\nu}}\geq\bm{0}$ satisfying the KKT conditions.
In particular, the stationarity condition is
\begin{equation}
    \bm{Q}\hat{\bm{x}}
    +\bm{c}
    +\tau\hat{\bm{p}}
    +\bm{A}^{\top}\hat{\bm{\lambda}}
    -\bm{W}^{\top}\hat{\bm{\nu}}
    =
    \bm{0}.
\end{equation}
Therefore,
\begin{equation}
    \hat{\bm{p}}
    =
    -\frac{1}{\tau}
    \left(
        \bm{Q}\hat{\bm{x}}
        +\bm{c}
        +\bm{A}^{\top}\hat{\bm{\lambda}}
        -\bm{W}^{\top}\hat{\bm{\nu}}
    \right).
    \label{eq:pstar}
\end{equation}

We next consider the subproblem with respect to $\bm{z}\in [0, 1]^n$.
To minimize its objective function (Eq.~\eqref{obj:fenchel}), the variables $z_i$ corresponding to the largest values of $\tau(\hat{p}_i)^2/4$ should be preferentially set to one.
Therefore, under the cardinality constraint, one optimal solution is obtained by setting $z_i=1$ for the $\kappa$ assets with the largest values of $\tau(\hat{p}_i)^2/4$.
Substituting Eq.~\eqref{eq:pstar} into $\tau(\hat{p}_i)^2/4$, we define the screening score for each asset $i\in[n]$ as
\begin{equation}
    \delta_i
    :=
    \frac{\tau}{4}(\hat{p}_i)^2
    =
    \frac{1}{4\tau}
    \left[
        \bm{Q}\hat{\bm{x}}
        +\bm{c}
        +\bm{A}^{\top}\hat{\bm{\lambda}}
        -\bm{W}^{\top}\hat{\bm{\nu}}
    \right]_i^2.
    \label{eq:screening-score}
\end{equation}
Let $\sigma: [n] \to [n]$ be a permutation satisfying
\begin{equation}
    \delta_{\sigma(1)}
    \geq
    \delta_{\sigma(2)}
    \geq
    \cdots
    \geq
    \delta_{\sigma(n)}.
    \label{eq:screening-score-order}
\end{equation}
If ties occur, their ordering is chosen arbitrarily.
We define the top-$\kappa$ score set and its complement as
\begin{equation}
    \mathcal{S}_{\kappa}
    :=
    \{\sigma(j) \mid j \in [\kappa]\},
    \quad
    \mathcal{S}_{\kappa}^{c}
    :=
    [n]\setminus\mathcal{S}_{\kappa}
    =
    \{\sigma(\kappa + j) \mid j \in [n-\kappa]\}.
    \label{eq:top-complement-sets}
\end{equation}

\subsection{Exclusion and Inclusion Rules}
When the Fenchel dual variable is fixed at $\bm{p}=\hat{\bm{p}}$, the $\bm{z}$-subproblem is expressed as
\begin{align}
    \min_{\bm{z}\in\mathcal{Z}}
    \quad&
    -\sum_{i=1}^{n}\delta_i z_i.
    \label{prob:z-score-subproblem}
\end{align}
This problem admits an optimal solution in which $z_i=1$ for the assets belonging to $\mathcal{S}_{\kappa}$, and its optimal value is
\begin{equation}
    -\sum_{j=1}^{\kappa}\delta_{\sigma(j)}.
    \label{eq:z-subproblem-optimal-value}
\end{equation}
Here, $\delta_{\sigma(\kappa)}$ is the smallest score in $\mathcal{S}_{\kappa}$, whereas $\delta_{\sigma(\kappa+1)}$ is the largest score in $\mathcal{S}_{\kappa}^{c}$.
Hence, when an asset is moved into $\mathcal{S}_{\kappa}$, the least costly adjustment is to remove the asset with score $\delta_{\sigma(\kappa)}$; when an asset is removed from $\mathcal{S}_{\kappa}$, it is to add the asset with score $\delta_{\sigma(\kappa+1)}$.

\begin{proposition}[Safe Screening Rules]
\label{prop:safe-screening-rules}
Suppose that $1\leq\kappa<n$.
Let $f_{\mathrm{LB}}$ be the optimal objective value of the convex relaxation problem~\eqref{relaxation}, and let $f_{\mathrm{UB}}$ be a valid upper bound on the original problem~\eqref{binary-reformulation}.
Define the exclusion set $\mathcal{I}_0$ and the inclusion set $\mathcal{I}_1$ as
\begin{align}
    \mathcal{I}_0
    &:=
    \left\{
        i\in\mathcal{S}_{\kappa}^{c}
        \;\middle|\;
        f_{\mathrm{LB}}
        +\delta_{\sigma(\kappa)}-\delta_i
        >
        f_{\mathrm{UB}}
    \right\},
    \label{eq:remove-set}
    \\
    \mathcal{I}_1
    &:=
    \left\{
        i\in\mathcal{S}_{\kappa}
        \;\middle|\;
        f_{\mathrm{LB}}
        +\delta_i-\delta_{\sigma(\kappa+1)}
        >
        f_{\mathrm{UB}}
    \right\}.
    \label{eq:keep-set}
\end{align}
Then, for every optimal solution $(\bm{x}^\star,\bm{z}^\star)$ of problem~\eqref{binary-reformulation},
\begin{equation}
    z_i^\star=0
    \quad(i\in \mathcal{I}_0),
    \qquad
    z_i^\star=1
    \quad(i\in \mathcal{I}_1).
\end{equation}
\end{proposition}

\begin{proof}
We first consider an asset $i\in\mathcal{S}_{\kappa}^{c}$.
When the constraint $z_i=1$ is imposed, asset $i$ replaces the asset corresponding to $\delta_{\sigma(\kappa)}$, which has the smallest score among the assets in $\mathcal{S}_{\kappa}$.
By evaluating the resulting change in the objective value of the $\bm{z}$-subproblem~\eqref{prob:z-score-subproblem}, we obtain the valid lower bound
\begin{equation}
f_{\text{LB}}
+
\delta_{\sigma(\kappa)}-\delta_i.
\end{equation}
If this lower bound exceeds $f_{\mathrm{UB}}$, no optimal solution to problem~\eqref{binary-reformulation} can satisfy $z_i=1$.
Thus, $z_i$ can be safely fixed to zero, which yields the exclusion rule.

Next, we consider an asset $i\in\mathcal{S}_{\kappa}$.
When the constraint $z_i=0$ is imposed, asset $i$ is replaced by the asset corresponding to $\delta_{\sigma(\kappa+1)}$.
Similarly, we obtain the valid lower bound
\begin{equation}
f_{\mathrm{LB}}
+
\delta_i-\delta_{\sigma(\kappa+1)}.
\label{eq:fix-zero-lower-bound}
\end{equation}
If this lower bound exceeds $f_{\mathrm{UB}}$, no optimal solution to problem~\eqref{binary-reformulation} can satisfy $z_i=0$.
Thus, $z_i$ can be safely fixed to one, which yields the inclusion rule.
\end{proof}

\subsection{Algorithm}
Algorithm~\ref{alg:safe-screening} summarizes our safe screening procedure for the cardinality-constrained portfolio optimization problem~\eqref{original}.
This algorithm computes a lower bound and screening scores from the convex relaxation problem~\eqref{relaxation} and compares them with an upper bound obtained from a feasible solution to the original problem~\eqref{binary-reformulation}, thereby safely fixing some of the binary variables.
Using the sets $\mathcal{I}_0$ and $\mathcal{I}_1$ obtained by Algorithm~\ref{alg:safe-screening}, we solve the reduced problem~\eqref{binary-soc} obtained by fixing
\begin{equation}
z_i=0\quad(i\in \mathcal{I}_0),
\qquad
z_i=1\quad(i\in \mathcal{I}_1)
\end{equation}
without losing the optimality guarantee.

\begin{algorithm}[t]
\caption{Safe Screening for Cardinality-Constrained Portfolio Optimization Problem~\eqref{original}}
\label{alg:safe-screening}
\begin{algorithmic}[1]
\State $\mathcal{I}_0\gets\emptyset,\ \mathcal{I}_1\gets\emptyset$.

\State Solve the convex relaxation problem
\eqref{relaxation-soc} and obtain its
optimal value $f_{\mathrm{LB}}$, a primal optimal solution
$(\hat{\bm{x}},\hat{\bm{z}})$, and optimal Lagrange multipliers
$\hat{\bm{\lambda}}$ and $\hat{\bm{\nu}}$.

\State Compute $\hat{\bm{p}}$ from Eq.~\eqref{eq:pstar} and calculate
$\delta_i$ for all $i\in[n]$ from Eq.~\eqref{eq:screening-score}.

\State Determine a permutation $\sigma$ satisfying
Eq.~\eqref{eq:screening-score-order}, and define
$\mathcal{S}_{\kappa}$ and $\mathcal{S}_{\kappa}^{c}$
according to Eq.~\eqref{eq:top-complement-sets}.

\State Construct a feasible solution to the original problem~\eqref{binary-reformulation} and let
its objective value be $f_{\mathrm{UB}}$.

\For{$i\in\mathcal{S}_{\kappa}^{c}$}
    \If{$f_{\mathrm{LB}}
        +\delta_{\sigma(\kappa)}-\delta_i
        >f_{\mathrm{UB}}$}
        \State $\mathcal{I}_0\gets \mathcal{I}_0\cup\{i\}$.
    \EndIf
\EndFor

\For{$i\in\mathcal{S}_{\kappa}$}
    \If{$f_{\mathrm{LB}}
        +\delta_i-\delta_{\sigma(\kappa+1)}
        >f_{\mathrm{UB}}$}
        \State $\mathcal{I}_1\gets \mathcal{I}_1\cup\{i\}$.
    \EndIf
\EndFor

\State \Return $(\mathcal{I}_0,\mathcal{I}_1)$.
\end{algorithmic}
\end{algorithm}

\section{Numerical Experiments}
This section presents the settings and results of the numerical experiments conducted to evaluate the effectiveness of our method.
All experiments were performed on a MacBook Air equipped with an Apple M3 processor and 24 GB of memory.
The optimization problems were solved using Gurobi Optimizer\footnote{\url{https://www.gurobi.com/}} 11.0.3.
\subsection{Experimental Setting}
\label{subsec:experimental-setting}
\begin{table*}[t]
\centering
\caption{Results for the S\&P 500 dataset.}
\label{tab:sp500_results}
\setlength{\tabcolsep}{4pt}
\begin{tabular}{cccrrrrrrrr}
\toprule
\multicolumn{3}{c}{Parameters}
&
\multicolumn{3}{c}{Screening}
&
\multicolumn{5}{c}{Time [s]}
\\
\cmidrule(lr){1-3}
\cmidrule(lr){4-6}
\cmidrule(lr){7-11}
$\kappa$
& $\tau$
& $r_{\min}$
& $|\mathcal{I}_0|$
& $|\mathcal{I}_1|$
& Rate
& $T_{\mathrm{scr}}$
& $T_{\mathrm{red}}$
& $T_{\mathrm{prop}}$
& $T_{\mathrm{dir}}$
& Speedup
\\
\midrule
\multirow[t]{9}{*}{40}
& \multirow[t]{3}{*}{0.001}
& 0.3\% & 0 & 15 & 3.7\% & 1.0 & 130.5 & 131.5 & 114.4 & 0.9$\times$ \\
& & 0.5\% & 0 & 31 & 7.7\% & 1.3 & 0.6 & 1.8 & 4.2 & 2.3$\times$ \\
& & 0.7\% & 0 & 11 & 2.7\% & 1.2 & 0.1 & 1.3 & 0.3 & 0.2$\times$ \\
\cmidrule(lr){2-11}
& \multirow[t]{3}{*}{0.005}
& 0.3\% & 301 & 10 & 77.6\% & 1.3 & 16.4 & 17.7 & 183.3 & 10.4$\times$ \\
& & 0.5\% & 353 & 35 & 96.8\% & 1.3 & 0.6 & 1.9 & 94.5 & 50.5$\times$ \\
& & 0.7\% & 0 & 15 & 3.7\% & 1.4 & 0.2 & 1.6 & 0.3 & 0.2$\times$ \\
\cmidrule(lr){2-11}
& \multirow[t]{3}{*}{0.010}
& 0.3\% & 303 & 10 & 78.1\% & 1.3 & 18.4 & 19.6 & 263.3 & 13.4$\times$ \\
& & 0.5\% & 350 & 36 & 96.3\% & 1.5 & 0.6 & 2.1 & 77.1 & 37.6$\times$ \\
& & 0.7\% & 0 & 15 & 3.7\% & 1.4 & 0.3 & 1.8 & 0.5 & 0.3$\times$ \\
\midrule
\multirow[t]{9}{*}{60}
& \multirow[t]{3}{*}{0.001}
& 0.3\% & 0 & 31 & 7.7\% & 1.6 & 54.3 & 56.0 & 27.4 & 0.5$\times$ \\
& & 0.5\% & 0 & 31 & 7.7\% & 1.4 & 0.5 & 1.9 & 2.0 & 1.1$\times$ \\
& & 0.7\% & 0 & 11 & 2.7\% & 1.4 & 0.2 & 1.6 & 0.3 & 0.2$\times$ \\
\cmidrule(lr){2-11}
& \multirow[t]{3}{*}{0.005}
& 0.3\% & 296 & 26 & 80.3\% & 1.9 & 6.9 & 8.9 & 194.5 & 21.9$\times$ \\
& & 0.5\% & 323 & 49 & 92.8\% & 1.6 & 1.5 & 3.1 & 34.4 & 11.1$\times$ \\
& & 0.7\% & 0 & 15 & 3.7\% & 1.5 & 0.2 & 1.7 & 0.4 & 0.2$\times$ \\
\cmidrule(lr){2-11}
& \multirow[t]{3}{*}{0.010}
& 0.3\% & 295 & 25 & 79.8\% & 1.7 & 6.0 & 7.7 & 127.5 & 16.5$\times$ \\
& & 0.5\% & 328 & 47 & 93.5\% & 1.4 & 1.1 & 2.5 & 42.7 & 16.9$\times$ \\
& & 0.7\% & 0 & 15 & 3.7\% & 1.5 & 0.3 & 1.8 & 0.5 & 0.3$\times$ \\
\midrule
\multirow[t]{9}{*}{80}
& \multirow[t]{3}{*}{0.001}
& 0.3\% & 0 & 45 & 11.2\% & 1.6 & 2.3 & 4.0 & 19.2 & 4.9$\times$ \\
& & 0.5\% & 0 & 30 & 7.5\% & 1.6 & 0.5 & 2.0 & 1.3 & 0.7$\times$ \\
& & 0.7\% & 0 & 11 & 2.7\% & 1.7 & 0.2 & 1.8 & 0.3 & 0.2$\times$ \\
\cmidrule(lr){2-11}
& \multirow[t]{3}{*}{0.005}
& 0.3\% & 293 & 56 & 87.0\% & 1.5 & 0.9 & 2.4 & 393.3 & 163.5$\times$ \\
& & 0.5\% & 0 & 55 & 13.7\% & 1.8 & 2.6 & 4.4 & 2.8 & 0.6$\times$ \\
& & 0.7\% & 0 & 15 & 3.7\% & 1.9 & 0.2 & 2.0 & 0.4 & 0.2$\times$ \\
\cmidrule(lr){2-11}
& \multirow[t]{3}{*}{0.010}
& 0.3\% & 294 & 57 & 87.5\% & 1.8 & 4.4 & 6.2 & 93.3 & 15.1$\times$ \\
& & 0.5\% & 311 & 75 & 96.3\% & 2.1 & 1.1 & 3.2 & 2.7 & 0.9$\times$ \\
& & 0.7\% & 0 & 15 & 3.7\% & 1.9 & 0.3 & 2.2 & 0.5 & 0.2$\times$ \\
\bottomrule
\end{tabular}
\end{table*}
We evaluated our screening rules using weekly returns of the constituent stocks of the S\&P 500 and Russell 2000 over the period from 2005 to 2024.
The returns were computed from adjusted closing prices obtained from Yahoo! Finance\footnote{\url{https://finance.yahoo.com/}}.
We retained assets with returns available for at least 90\% of the aforementioned periods and then removed periods containing any missing values.
After this preprocessing, the S\&P 500 dataset contained $n=401$ assets and $T=946$ observations,
whereas the Russell 2000 dataset contained $n=812$ assets and $T=942$ observations.

Let $\bm{\mu}\in\mathbb{R}^{n}$ and $\bm{\Sigma}\in\mathbb{S}_{+}^{n}$ denote the sample mean vector and sample covariance matrix, respectively, computed from the weekly return observations.
The parameter $r_{\min}$ denotes the minimum required weekly return.
For each dataset, we considered the following cardinality-constrained mean--variance portfolio optimization model:
\begin{subequations}\label{experiment-model}
\begin{align}
\min_{\bm{x}\in\mathbb{R}^{n}}
\quad&
\frac{1}{2}\bm{x}^{\top}\bm{\Sigma}\bm{x}
+
\tau\|\bm{x}\|_{2}^{2}
\label{prob:numerical-model}
\\
\mathrm{s.~t.}\quad
&\|\bm{x}\|_{0}\le\kappa,\label{cons:ex-card}\\
&\bm{1}^{\top}\bm{x}=1,\label{cons:ex-budget}\\
&\bm{\mu}^{\top}\bm{x}\ge r_{\min},\label{cons:ex-minreturn}\\
&\bm{x}\ge\bm{0}.\label{cons:ex-longonly}
\end{align}
\end{subequations}
Eq.~\eqref{cons:ex-budget} requires the investment weights across all assets to sum to one.
Eq.~\eqref{cons:ex-minreturn} requires the expected return of the portfolio to be at least the minimum required return.
Eq.~\eqref{cons:ex-longonly} is the nonnegativity constraint, which prohibits short selling. 

We tested all combinations of $\tau\in\{0.001,0.005,0.010\}$ for the regularization parameter, $r_{\min}\in\{0.3\%,0.5\%,0.7\%\}$ for the minimum required return, and the cardinality parameter $\kappa\in\{40,60,80\}$ for the S\&P 500 dataset and $\kappa\in\{80,120,160\}$ for the Russell 2000 dataset.
These cardinality limits correspond to approximately 10\%, 15\%, and 20\% of the available assets, yielding 27 problem instances per dataset.
A feasible upper bound $f_{\text{UB}}$ was constructed by selecting the $\kappa$ largest components of convex relaxation solution $\hat{\bm z}$, replacing low-$\hat{z}_i$ assets with assets satisfying $\mu_i\ge r_{\min}$ when necessary, and re-optimizing the weights of the selected assets.

We used the following evaluation metrics:
\begin{itemize}
    \item \textbf{Number of excluded assets ($|\mathcal{I}_0|$)}:
    the number of assets fixed to zero by our screening rules.

    \item \textbf{Number of included assets ($|\mathcal{I}_1|$)}:
    the number of assets fixed to one by our screening rules.

    \item \textbf{Screening rate}:
    the percentage of assets fixed by our screening rules, defined as
    \begin{equation}
        \mathrm{Rate}
        :=
        \frac{|\mathcal{I}_0|+|\mathcal{I}_1|}{n}\times100\%.
        \label{eq:screening-rate}
    \end{equation}

    \item \textbf{Screening time ($T_{\mathrm{scr}}$)}:
    the time required to solve the convex relaxation problem~\eqref{relaxation-soc}, construct the upper bound, compute the screening scores, and apply the variable-fixing rules.

    \item \textbf{Reduced solution time ($T_{\mathrm{red}}$)}:
    the time required to solve the reduced problem~\eqref{binary-soc} after applying the screening rules.

    \item \textbf{Total proposed time ($T_{\mathrm{prop}}$)}:
    the total computation time of our method, defined as $T_{\mathrm{prop}}
        :=
        T_{\mathrm{scr}}+T_{\mathrm{red}}$.

    \item \textbf{Direct solution time ($T_{\mathrm{dir}}$)}:
    the time required to solve problem~\eqref{binary-soc} without screening.

    \item \textbf{Speedup}:
    the ratio of the direct solution time to the total proposed time, defined as $T_{\mathrm{dir}}/{T_{\mathrm{prop}}}$.
\end{itemize}
A time limit of 1800 seconds was imposed on each exact optimization problem~\eqref{binary-soc}.
\begin{table*}[t]
\centering
\caption{Results for the Russell 2000 dataset.}
\label{tab:russell2000_results}
\setlength{\tabcolsep}{4pt}
\begin{tabular}{cccrrrrrrrr}
\toprule
\multicolumn{3}{c}{Parameters}
&
\multicolumn{3}{c}{Screening}
&
\multicolumn{5}{c}{Time [s]}
\\
\cmidrule(lr){1-3}
\cmidrule(lr){4-6}
\cmidrule(lr){7-11}
$\kappa$
& $\tau$
& $r_{\min}$
& $|\mathcal{I}_0|$
& $|\mathcal{I}_1|$
& Rate
& $T_{\mathrm{scr}}$
& $T_{\mathrm{red}}$
& $T_{\mathrm{prop}}$
& $T_{\mathrm{dir}}$
& Speedup
\\
\midrule
\multirow[t]{9}{*}{80}
& \multirow[t]{3}{*}{0.001}
& 0.3\% & 0 & 25 & 3.1\% & 5.7 & 308.4 & 314.1 & $>$1800.0 & $>$5.7$\times$ \\
& & 0.5\% & 0 & 44 & 5.4\% & 5.4 & 15.2 & 20.6 & 1291.0 & 62.7$\times$ \\
& & 0.7\% & 0 & 46 & 5.7\% & 5.5 & 6.4 & 12.0 & 8.0 & 0.7$\times$ \\
\cmidrule(lr){2-11}
& \multirow[t]{3}{*}{0.005}
& 0.3\% & 0 & 3 & 0.4\% & 6.4 & $>$1800.0 & $>$1800.0 & $>$1800.0 & \multicolumn{1}{c}{N/A} \\
& & 0.5\% & 0 & 16 & 2.0\% & 4.9 & $>$1800.0 & $>$1800.0 & $>$1800.0 & \multicolumn{1}{c}{N/A} \\
& & 0.7\% & 0 & 38 & 4.7\% & 5.3 & $>$1800.0 & $>$1800.0 & $>$1800.0 & \multicolumn{1}{c}{N/A} \\
\cmidrule(lr){2-11}
& \multirow[t]{3}{*}{0.010}
& 0.3\% & 0 & 0 & 0.0\% & 5.4 & $>$1800.0 & $>$1800.0 & $>$1800.0 & \multicolumn{1}{c}{N/A} \\
& & 0.5\% & 0 & 7 & 0.9\% & 5.1 & $>$1800.0 & $>$1800.0 & $>$1800.0 & \multicolumn{1}{c}{N/A} \\
& & 0.7\% & 0 & 22 & 2.7\% & 5.9 & $>$1800.0 & $>$1800.0 & $>$1800.0 & \multicolumn{1}{c}{N/A} \\
\midrule
\multirow[t]{9}{*}{120}
& \multirow[t]{3}{*}{0.001}
& 0.3\% & 0 & 35 & 4.3\% & 5.8 & 9.2 & 15.0 & 11.3 & 0.8$\times$ \\
& & 0.5\% & 0 & 46 & 5.7\% & 6.2 & 7.1 & 13.2 & 10.1 & 0.8$\times$ \\
& & 0.7\% & 0 & 45 & 5.5\% & 6.2 & 5.7 & 11.9 & 9.6 & 0.8$\times$ \\
\cmidrule(lr){2-11}
& \multirow[t]{3}{*}{0.005}
& 0.3\% & 0 & 37 & 4.6\% & 7.8 & $>$1800.0 & $>$1800.0 & $>$1800.0 & \multicolumn{1}{c}{N/A} \\
& & 0.5\% & 0 & 67 & 8.3\% & 6.5 & 850.3 & 856.8 & 1272.9 & 1.5$\times$ \\
& & 0.7\% & 0 & 77 & 9.5\% & 8.3 & 10.3 & 18.7 & 62.1 & 3.3$\times$ \\
\cmidrule(lr){2-11}
& \multirow[t]{3}{*}{0.010}
& 0.3\% & 0 & 22 & 2.7\% & 11.1 & $>$1800.0 & $>$1800.0 & $>$1800.0 & \multicolumn{1}{c}{N/A} \\
& & 0.5\% & 619 & 63 & 84.0\% & 5.8 & 83.0 & 88.7 & 770.2 & 8.7$\times$ \\
& & 0.7\% & 0 & 68 & 8.4\% & 12.1 & 576.4 & 588.5 & 1482.8 & 2.5$\times$ \\
\midrule
\multirow[t]{9}{*}{160}
& \multirow[t]{3}{*}{0.001}
& 0.3\% & 0 & 36 & 4.4\% & 11.6 & 10.1 & 21.7 & 11.0 & 0.5$\times$ \\
& & 0.5\% & 0 & 46 & 5.7\% & 11.8 & 6.7 & 18.5 & 226.2 & 12.2$\times$ \\
& & 0.7\% & 0 & 46 & 5.7\% & 10.9 & 5.8 & 16.7 & 19.0 & 1.1$\times$ \\
\cmidrule(lr){2-11}
& \multirow[t]{3}{*}{0.005}
& 0.3\% & 0 & 93 & 11.5\% & 17.6 & 147.1 & 164.7 & 735.5 & 4.5$\times$ \\
& & 0.5\% & 0 & 90 & 11.1\% & 12.4 & 11.3 & 23.6 & 12.8 & 0.5$\times$ \\
& & 0.7\% & 0 & 81 & 10.0\% & 12.5 & 11.1 & 23.6 & $>$1800.0 & $>76.3\times$ \\
\cmidrule(lr){2-11}
& \multirow[t]{3}{*}{0.010}
& 0.3\% & 0 & 88 & 10.8\% & 18.0 & 1010.0 & 1028.0 & 1152.7 & 1.1$\times$ \\
& & 0.5\% & 0 & 102 & 12.6\% & 10.9 & 436.0 & 446.8 & $>$1800.0 & $>4.0\times$ \\
& & 0.7\% & 0 & 106 & 13.1\% & 13.3 & 12.0 & 25.4 & 12.8 & 0.5$\times$ \\
\bottomrule
\end{tabular}
\end{table*}
\subsection{Results for the S\&P 500 dataset}
Table~\ref{tab:sp500_results} shows the results for the S\&P 500 dataset.
For this dataset, our screening rules were particularly effective when the regularization parameter $\tau$ was moderate or large.
With \(\tau\in\{0.005,0.010\}\) and \(r_{\min}\in\{0.3\%,0.5\%\}\), the screening rate was high in most cases, ranging from 77.6\% to 96.8\%. 
In these cases, many assets were safely fixed, and this led to substantial reductions in computation time.
For example, when $\kappa=80$, $\tau=0.005$, and $r_{\min}=0.3\%$, our method fixed the weights at zero for 293 assets and at one for 56 assets, reducing the total computation time from 393.3 seconds to 2.4 seconds.
This corresponds to a speedup rate of 163.5 times.

The results also show a dependence on the return requirement.
For $r_{\min}=0.3\%$ and $0.5\%$, our rules often achieved high screening rates and large speedups, especially for $\tau=0.005$ and $\tau=0.010$.
In contrast, when $r_{\min}=0.7\%$, the screening rate remained low for all values of $\kappa$ and $\tau$; however, note that these problem instances were solved directly within one second even without screening.
Overall, the largest improvements were observed on computationally challenging instances, indicating that our screening rules can be particularly effective when the original problem is difficult to solve.
\subsection{Results for the Russell 2000 dataset}
\label{subsec:russell-results}
Table~\ref{tab:russell2000_results} shows the results for the Russell 2000 dataset.
For this dataset, the screening and computational results varied substantially across parameter settings.
Compared with the results for the S\&P 500 dataset, exclusion fixings were rare, whereas inclusion fixings occurred in most parameter settings. 
Thus, the two types of screening rules exhibited markedly different behavior on this dataset.

Despite the lower screening rates in many cases, our method achieved substantial improvements in several difficult problem instances.
For example, when $\kappa=80$, $\tau=0.001$, and
$r_{\min}=0.5\%$, only 5.4\% of the variables were fixed, but the total computation time was reduced from 1291.0 seconds to 20.6 seconds, corresponding to a speedup of 62.7 times.
Moreover, when $(\kappa,\tau,r_{\min})=(160,0.005,0.7\%)$ and $(160,0.010,0.5\%)$, the direct solution reached the time limit, whereas our method completed in 23.6 and 446.8 seconds, respectively.
The corresponding screening rates were 10.0\% and 12.6\%.
These results indicate that even a moderate screening rate can substantially accelerate computation and enable some instances that are not solved within the time limit by direct optimization to be solved to optimality.

Although the screening rate did not increase monotonically with the cardinality limit in every parameter setting, an overall tendency toward higher screening rates was observed as $\kappa$ increased.
Overall, although the effectiveness of our method varied across the Russell 2000 instances, it provided substantial improvements in several computationally challenging cases.

\section{Conclusion}
We proposed safe screening rules for cardinality-constrained portfolio optimization problems with convex quadratic objective functions and linear constraints.
Specifically, by applying the perspective relaxation and deriving its Fenchel dual, we obtained screening scores that incorporate the Lagrange multipliers associated with the linear constraints.
Furthermore, by comparing a lower bound obtained from the convex relaxation problem with an upper bound obtained from a feasible solution to the original problem, we established exclusion and inclusion rules that can safely fix binary variables to either $0$ or $1$ without eliminating any optimal solution of the original problem.

The computational results demonstrated that our screening rules can substantially improve exact solution procedures for cardinality-constrained portfolio optimization.
The improvements were especially pronounced in challenging instances with moderate or strong regularization and less restrictive return requirements.
For the S\&P 500 dataset, the rules often fixed a large fraction of asset-selection variables, leading to significant speedups.
For the Russell 2000 dataset, although the screening effect was more heterogeneous, the reduced problems were solved to global optimality in several cases where direct optimization reached the time limit.
Thus, our framework offers a promising optimality-preserving preprocessing technique for large-scale cardinality-constrained portfolio optimization.

The number of variables that can be fixed by the screening rules,
however, depends on the tightness of the lower bound obtained from the
convex relaxation problem and the quality of the upper bound obtained
from a feasible solution.
Future work therefore includes developing stronger convex relaxations and efficient heuristics for constructing high-quality feasible solutions.
Beyond improving these bounds, another important direction is to develop more efficient screening methods that achieve stronger variable reduction while controlling the associated computational overhead.
One possible approach is to extend the screening-based cut-generation and cut-selection framework~\cite{tan2026fast} to the linearly constrained portfolio optimization setting considered in this study.

\bibliographystyle{ACM-Reference-Format}
\bibliography{references}

\end{document}